\pdfoutput=1
\documentclass{amsart}

\usepackage{amsmath }
\usepackage{bbm,dsfont}
\usepackage{mathrsfs}
\usepackage[all]{xy}
\usepackage{amsfonts}
\usepackage{ifthen}
\usepackage{rotating}
\usepackage{amssymb}

\usepackage{tikz}

\usepackage[pdftex]{hyperref}

\theoremstyle{plain}
\newtheorem{lemma}{Lemma}[section]
\newtheorem{teo}[lemma]{Theorem}
\newtheorem{propo}[lemma]{Proposition}

\newtheorem{definition}[lemma]{Definition}

\theoremstyle{definition}

\newtheorem{claim*}{Claim}

\theoremstyle{remark}
\newtheorem{remark}[lemma]{Remark}

\newcommand{\Cox}{\operatorname{Cox}}

\newcommand{\Pic}{\operatorname{Pic}}

\newcommand{\Sp}{\operatorname{Spec}}
\newcommand{\Proj}{\operatorname{Proj}}

\newcommand{\Nhom}{\operatorname{N}}
\newcommand{\Nef}{\operatorname{Nef}}
\newcommand{\Rsec}{\operatorname{R}}
\newcommand{\NE}{\operatorname{NE}}
\newcommand{\Hom}{\operatorname{Hom}}
\newcommand{\Neo}{\operatorname{\overline{NE}^1}}
\newcommand{\Aone}{\operatorname{A_1}}

\newcommand{\pp}{\mathbb{P}}
\newcommand{\qq}{\mathbb{Q}}
\newcommand{\rr}{\mathbb{R}}
\newcommand{\cc}{\mathbb{C}}
\newcommand{\zz}{\mathbb{Z}}
\newcommand{\nn}{\mathbb{N}}

\newcommand{\rmap}{\dashrightarrow}

\begin{document}

\title{\bf On embeddings of Mori dream spaces}

\author{John Levitt}
\address{
Department of Mathematics \newline
Occidental College \newline
Los Angeles, CA}
\email{levitt@oxy.edu}

\begin{abstract}
General criteria are given for when an embedding of a Mori dream space into another satisfies certain nice combinatorial conditions on some of their associated cones. An explicit example of such an embedding is studied.

\end{abstract}
\maketitle

\section{Introduction}

Mori dream spaces were introduced by Hu and Keel in ~\cite{hk} as varieties for which the minimal model program can be run in a very general way.  They also are the varieties for which the Cox ring is finitely generated.  Using a presentation of this ring, they showed that a Mori dream space $X$ always embeds into a toric variety $W$ in such a way that the combinatorics of $W$ describes much of the birational geometry of $X$.  In this paper we give simple criteria for when embeddings between Mori dream spaces behave this way, and an explicit example of a nontrivial such embedding for a nontoric del Pezzo surface.  Further background on the connection between Mori dream spaces and their Cox rings can be found in the excellent surveys ~\cite{lv} and ~\cite{mck}, which also give much more extensive references to the current literature.

The author would like to acknowledge and thank Paolo Cascini and James M$^c$Kernan for their helpful suggestions on my work.


\section{Mori dream spaces}

In this section we define our main objects of study, and discuss certain embeddings.  

\begin{definition} A normal projective variety $X$ is a \textbf{Mori Dream Space} if

    \begin{enumerate}
        \item $X$ is $\qq$-factorial and $\Pic{(X)}_{\qq}=\Nhom^1 {(X)}$,
        
        \item $\Nef{(X)}$ is the affine hull of finitely many semiample divisors,
        
        \item there is a finite collection of birational maps $f_i:X\rmap X_i$ which are isomorphisms in codimension one, such that when $D$ is a movable divisor on $X$, then there is an $i$ such that $D=f_i^*D_i$ for some semiample divisor $D_i$ on $X_i$.

    \end{enumerate}

\end{definition}

Examples of such spaces are all Fano varieties ~\cite{bchm}
and all toric varieties ~\cite{hk}.
\\

These spaces have a number of nice properties, and some interesting related cone structures.  Perhaps the most important property and hence their name is that the minimal model program can be run and will terminate using not just the canonical divisor to test for nefness, but any divisor ~\cite{hk}. 
Turning to the cone structures, there is a chamber structure on the pseudo-effective of divisors of a Mori dream space, which parametrizes the rational maps from it.

\begin{definition} Given a line bundle $L$ on a scheme $X$, the \textbf{section ring} is $$\Rsec_L:=\bigoplus_{n \in \nn}H^0(X,L^{\otimes n}).$$

\end{definition}

Note that when $D$ is effective and $\Rsec_D$ is finitely generated, there is an induced rational map $$f_D:X\rmap \Proj{(\Rsec_D)}.$$ We say that two such divisors, $D_1$ and $D_2$, are \textbf{Mori equivalent} if the rational maps $f_{D_i}$ have isomorphic images making the natural triangular diagram commute.  The equivalence classes, which we call \textbf{Mori chambers}, impose a natural chamber structure on $\overline{\NE}^1(X)$. For Mori dream spaces, these chambers are the same as the GIT chambers, equivalence classes of effective divisors whose linearizations give the same GIT quotient. 
These chamber structures are tied to a single ring:

\begin{definition} Suppose that the classes of the line bundles $L_1,...,L_n$ are a basis of $\Pic{(X)}$. The \textbf{Cox ring} of $X$ is defined:

$$\Cox{(X,L_1,...,L_n)}:=\bigoplus_{(m_1,...,m_n) \in \zz^n}H^0(X,L_1^{\otimes m_1}\otimes ... \otimes L_n^{\otimes m_n})$$

\end{definition}

 Note that this ring is graded naturally by $\Pic{(X)}$. Furthermore, when the choice of basis does not matter, such as for determining finite generation, we use the notation $\Cox{(X)}$.  We now get the following criterion for determining when we have a Mori dream space: \\
 
 \begin{teo} ~\cite{hk} Given a $\qq$-factorial projective variety $X$ with $\Pic{(X)}_{\qq}=\Nhom^1{(X)}$, $X$ is a Mori dream space iff $\Cox{(X)}$ is finitely generated.
 
 \end{teo}

  In the case of a Mori dream space the Cox ring, along with a line bundle, tells us about the small birational type of $X$. To see this, note that the Cox ring has an ideal that plays the role of the irrelevant ideal.  For a line bundle $L$, let $J_L:=\sqrt{(\Rsec_L)}$. This ideal determines the non-semistable points in $\Sp{(\Cox{(X)})}$.  \\
 
 \begin{teo}  ~\cite{hk},~\cite{lv}, A Mori dream space $X$ is a good geometric quotient of $\Sp{(\Cox{(X)})}-J_L $ by the torus $\Hom{(\Pic{(X)},k^*)}$, where $L$ is any ample line bundle on $X$.
 
 \end{teo}

As we vary the line bundle across the different top-dimensional GIT chambers of $X$, we get varieties with the same Cox ring, which are isomorphisms in codimension one, such as flips.  Thus, we can recover $X$ and its flips by understanding $\Cox{(X)}$ and $\Pic{(X)}$.  An interesting consequence of this is that we can identify embeddings of $X$ that respect the chamber structure.

\begin{definition} Suppose $X$ and $Y$ are Mori dream spaces.  An embedding $ X\subset Y$ is called a \textbf{Mori embedding} if it satisfies the following:

    \begin{enumerate}
        \item The restriction of $\Pic{(W)}_{\qq}$ to $\Pic{(X)}_{\qq}$ is an isomorphism.
        \item There is an induced morphism $\overline{\NE}^1(W) \rightarrow \overline{\NE}^1(X)$.
        \item Every Mori chamber of $X$ is a finite union of Mori chambers of $W$.
        \item Given a contraction $f:X\rmap X'$, there is a rational contraction $\hat{f}:W\rmap X'$ which restricts to $f$.
    
    \end{enumerate}

\end{definition}

Nontrivial such embeddings always exist for non-toric Mori dream spaces:

\begin{propo} ~\cite{hk} Given a Mori dream space $X$, there exists quasi-smooth projective toric variety $W$ and a Mori embedding $X\subset W$.

\end{propo}

    Lastly, we can give alternate criteria for determining when an embedding is a Mori embedding.
    
\begin{teo} \label{prop:embed} Suppose $X$ and $W$ are Mori Dream Spaces and
$\Phi:X\rightarrow W$ is an embedding such that
    \begin{enumerate}
        \item the restriction of $\Pic{(W)}_{\qq}$ to
        $\Pic{(X)}_{\qq}$ is an isomorphism, and
        \item the induced map $\Phi^*:\Cox{(W)}\rightarrow \Cox{(X)}$ is surjective.
    \end{enumerate}
Then $\Phi$ is a Mori embedding.
\end{teo}

\begin{proof} First, we show the existence of an isomorphism
$\Neo{(W)}\rightarrow \Neo{(X)}$ after restriction.  Injectivity is an
immediate consequence of (1). To show surjectivity, suppose $D$ is a
representative of a class in $\Neo{(X)}$.  By definition, there
exists some $m$ large enough such that $H^{0}(X,mD)\neq 0$.  Since
the induced map between Cox rings is surjective by (2), there must
exist a $D_W$ on $W$ which restricts to $D$ and such that
$H^{0}(W,kD_W)\neq 0$ for some large enough $k$.  Thus, the class of
$D_W$ is in $\Neo{(W)}$.

    Second, we show that a Mori chamber of $X$ is a finite union of
Mori chambers of $W$.  By definition of a Mori Dream Space, there
are only finitely many chambers for $W$, thus it suffices to show
that any chamber for $W$ is contained in a chamber for $X$.  Suppose
that $D$ is an effective divisor in a Mori chamber $C$ of $W$. Since
its restriction $D_X$ is effective, finite generation of the
corresponding section rings implies the following diagram commutes,
where vertical arrows represent inclusions:

\begin{center}
\begin{tikzpicture}[scale=.8]
\draw (0,0) node (a){$W$};
\draw (4,0) node (b){$\Proj{(R(W,D))}$};
\draw (0,-2) node (c){$X$};
\draw (4,-2) node (d){$\Proj{(R(X,D_X))}$};
\draw [->, dashed] (a)--(b);
\draw [->] (c)--(a);
\draw [->, dashed] (c)--(d);
\draw [->] (d)--(b);
\end{tikzpicture}
\end{center}

Note that (2) implies that the rightmost vertical arrow is a closed
embedding. Now consider an effective divisor $D^{\prime}$ on $W$,
Mori equivalent to $D$, and restricting to $D^{\prime}_X$. Since by
definition $\Proj{(R(W,D))}\cong \Proj{(R(W,D^{\prime}))}$, and the
closed embeddings given by the right-hand side of the diagram are
induced by the restriction of global sections of the corresponding
section rings, then up to isomorphism the models given by $D_X$ and
$D^{\prime}_X$ are the same.

    Lastly, the restriction of contractions property of a Mori
embedding is an immediate consequence of the definition of Mori
Dream spaces, and the containment property of Mori Cones above.

\end{proof}

\section{Example}

In this section we construct a nontrivial example of a Mori embedding.  We consider the case of embedding the simplest smooth, non-toric del Pezzo surface in a toric variety found by considering a presentation of the Cox ring of the surface.  \\

    Fixing notation, let $\pi: X_4 \rightarrow \pp^2$ be the
blow-up of four points in general linear position, and denote by
$h:=\pi^*l$ the pullback of a line, and by $l_i$ the exceptional
divisor corresponding to a blown up point.  Recall that these
divisors form a basis for $\Pic{(X_4)}=\zz^5$.  In ~\cite{bp} and ~\cite{ct} the generators for the Cox rings of smooth del Pezzo surfaces are given.  With this choice of basis the
generators of $\Cox{(X_4)}$ are given by the the ten exceptional
divisors of the forms $l_i$ and $h-l_i-l_j$ for $i\neq j$. Let $R$
be the polynomial ring with 10 variables
$\cc [x_1,...,x_{10}]$, and consider the surjection $R
\rightarrow \Cox{(X_4)}$ given by $x_k\mapsto g_k$, where each $g_k$
corresponds to one of the ten exceptional curves as follows:
$$g_1:=h-l_1-l_2,\text{ }g_2:=h-l_1-l_3,\text{ } g_3:=h-l_1-l_4,$$ $$g_4:=h-l_2-l_3,\text{ }
g_5:=h-l_2-l_4,\text{ } g_6:=h-l_3-l_4,$$ $$g_7:=l_1,\text{ }
g_8:=l_2, \text{ }g_9:=l_3, \text{ and }g_{10}:=l_4.$$

Since $R$ is polynomial, it is the Cox ring of a toric variety $W$ ~\cite{COX}, ~\cite{hk}.

    To determine $W$ explicitly, we construct its 1-skeleton $\Delta^1(W)$ by
considering a torus action on $\Sp{(R)}$ compatible with the torus
action on $\Cox{(X_4)}$ which gives us $X_4$ as a GIT quotient.  This
torus is given by $G:=\Hom{(A_1(X_4),\cc^*)}$, noting that the
Chow group and Picard group are the same in this case.  Since $R$ is
a polynomial ring in 10 variables, $\Sp{(R)}=\cc^{10}$, and
$\Delta^1(\Sp{(R)})$ is given by the 10 standard coordinate vectors
in $\rr^{10}$.  Taking the GIT quotient $\Sp{(R)}//G$ imposes
linear relations on the torus-invariant divisors of $\Sp{(R)}$.
These are given by the weights of the natural torus action of
multiplication specified by $\Pic{(X_4)}$.  The images of the
$T$-invariant divisors of $\Sp{(R)}$ under this action form the
1-skeleton for $W:=\Sp{(R)}//G$.

    Let $N\simeq \zz^5$ be the lattice containing the fan of $W$, and $M=\Hom{(N,\zz)} \simeq \zz^5$ be its dual lattice. Since $X_4$ is a Mori Dream Space, we have
$\Aone{(X_4)}=\Aone{(W)}$, and hence the compatible torus action can be
understood by considering the exact sequence of abelian groups:
\[0\rightarrow M\rightarrow \zz^{\Delta(1)}\rightarrow \Aone{(W)}\rightarrow 0\]

\noindent where $\zz^{\Delta(1)}=\zz^{10}$ has a canonical basis given by the ten
exceptional divisors $E_1,E_2, \dots ,E_{10}$. Thus, $\zz^{10} \rightarrow \Aone{(W)}
$ is represented by the matrix  
$$ \left(
           \begin{array}{rrrrrrrrrr}
             1 & 1 & 1 & 1 & 1 & 1 & 0 & 0 & 0 & 0 \\
             -1 & -1 & -1 & 0 & 0 & 0 & 1 & 0 & 0 & 0 \\
             -1 & 0 & 0 & -1 & -1 & 0 & 0 & 1 & 0 & 0 \\
             0 & -1 & 0 & -1 & 0 & -1 & 0 & 0 & 1 & 0 \\
             0 & 0 & -1 & 0 & -1 & -1 & 0 & 0 & 0 & 1 \\
                       \end{array}
         \right).$$

Computing the kernel of this matrix gives us the columns of the matrix $A$ representing $M\rightarrow \zz^{\Delta(1)}$. Taking the $\zz$-dual of the above sequence describes the torus
action on lattices via:

\[0\rightarrow \Aone{(W)}^{*} \rightarrow \zz^{10}\rightarrow N \rightarrow 0\]

To work on the level of fans we tensor the above sequence by
$\rr$, giving us maps between real vector spaces where the
injective map $f:\rr^5\rightarrow \rr^{10}$ given in
the exact sequence is still represented by 

$$A^T= \left(
          \begin{array}{rrrrrrrrrr}
            1 & 0 & 0 & 0 & 0 & -1 & 1 & 1 & -1 & -1 \\
            0 & 1 & 0 & 0 & 0 & -1 & 1 & 0 & 0 & -1 \\
            0 & 0 & 1 & 0 & 0 & -1 & 1 & 0 & -1 & 0 \\
            0 & 0 & 0 & 1 & 0 & -1 & 0 & 1 & 0 & -1 \\
            0 & 0 & 0 & 0 & 1 & -1 & 0 & 1 & -1 & 0 \\
          \end{array}
        \right).$$

Each column above represents the coordinates of a $T$-invariant
divisor on $W$. Unfortunately, the 1-skeleton is not sufficient to
describe $\Delta(W)$; Cox rings do not uniquely determine the fan of
a variety. To proceed, we select a natural toric variety with this
1-skeleton, and show that we can obtain a Mori embedding for this
choice.

Following Example 2.11 of ~\cite{lv}, consider the ample divisor $D=11h-5l_1-3l_2-2l_3-l_4$ on $X_4$.  We find the irrelevant ideal $J_D$ by taking the radical of the monomial ideal generated by a basis of $\Cox{(X_4)}_D$. This can be done in Macaulay 2:\\

{\small \begin{verbatim}
i1 :S=QQ[a,b,c,d,e,f,x,y,z,w,Degrees=>{{1,-1,-1,0,0},{1,-1,0,-1,0},
{1,-1,0,0,-1},{1,0,-1,-1,0},{1,0,-1,0,-1},{1,0,0,-1,-1},{0,1,0,0,0},
{0,0,1,0,0},{0,0,0,1,0},{0,0,0,0,1}},Heft=>{3,1,1,1,1}]

i2 : radical monomialIdeal basis({11,-5,-3,-2,-1},S)
\end{verbatim} }

This gives us 42 monomials, each consisting of 5 variables out of the 10.  These relations determine the cone structure of a $\qq$-factorial toric variety $W$ with the above 1-skeleton.  Since all the cones must be simplicial in this case, we conclude from the fan structure that $W$ must be $\pp^5$ blown up along four $T$-invariant copies of $\pp^2$.

\begin{remark}
Choosing $D$ to be the anti-canonical divisor, we can modify the above script to get a different set of cones.  We can check the combinatorics in Magma via:

{\small \begin{verbatim}
R<a,b,c,d,e,f,x,y,z,w> := PolynomialRing(Rationals(),10);
I := [ ideal<R|a*b*c*x, a*d*e*y, a*c*d*x*y, a*b*e*x*y, a*f*x*y, b*d*f*z, 
       b*c*d*x*z, b*e*x*z, a*b*f*x*z, c*d*y*z, b*d*e*y*z, a*d*f*y*z, c*e*f*w,
        c*d*x*w, b*c*e*x*w, a*c*f*x*w, b*e*y*w, c*d*e*y*w, a*e*f*y*w, a*f*z*w,
        c*d*f*z*w, b*e*f*z*w>];
Z := [[1,1,1,1,1,1,0,0,0,0],[-1,-1,-1,0,0,0,1,0,0,0],[-1,0,0,-1,-1,0,0,1,0,0],
[0,-1,0,-1,0,-1,0,0,1,0],[0,0,-1,0,-1,-1,0,0,0,1] ];
Q := [];
C := CoxRing(R,I,Z,Q);
X := ToricVariety(C);
IsQFactorial(X);
IsProjective(X);
IsComplete(X);
\end{verbatim} }

This gives us a projective toric variety with the same 1-skeleton as before, but has 22 cones in its fan, and is not $\qq$-factorial.   Since projectivity  depends solely on the 1-skeleton, this shows that all the toric varieties with the same Cox ring as $W$ are projective, correcting the example in ~\cite{lv}.

\end{remark}

    Let $x_0,...,x_5$ be coordinates on $\pp^5$.  Under the
standard torus action of scalar multiplication, the $T$-invariant
planes are given by those planes which have three nonzero
coordinates. Without loss of generality consider:

\begin{itemize}

 \item[] \begin{center}$\Sigma_1=\{x_0=x_3=x_5=0\}$\end{center}
 \item[] \begin{center}$\Sigma_2=\{x_0=x_2=x_4=0\}$\end{center}
 \item[] \begin{center}$\Sigma_3=\{x_1=x_2=x_3=0\}$\end{center}
 \item[] \begin{center}$\Sigma_4=\{x_1=x_4=x_5=0\}$\end{center}

\end{itemize}

\begin{lemma} \label{lemma:plane4} There exists a 2-plane in $\pp^5$ that intersects
the four 2-planes $\Sigma_i$ each in a point not lying on the
others.
\end{lemma}
\begin{proof}
Consider the subvariety of the Grassmannian $\textbf{G}(2,5)$,
$\Gamma=\{P\in\textbf{G}(2,5) :P\cap \Sigma_i \neq \emptyset,
\forall i\}$.  Since $\dim{\textbf{G}(2,5)}=9$ and each intersection
with a $\Sigma_i$ imposes one condition, then as long as $\Gamma$ is
nonempty, $\dim{\Gamma} \geq 5$.  But $\Gamma \neq \emptyset$, as
any plane containing the line through $\Sigma_1\cap\Sigma_2$ and
$\Sigma_3\cap\Sigma_4$ is in this variety.

Next, let $U$ be the subset of $\Gamma$ containing planes of the
desired type.  To show this is a non-empty, open subvariety of
$\Gamma$, we claim that its complement is a finite union of closed
subvarieties of dimension no greater than four.  There are two
general cases to consider:

\begin{itemize}

    \item \textit{Planes $P$ which intersect $\Sigma_i$ in a line $l$ (for a fixed
    $i$.)}  Note $P$ is given by $l$ and a point $p$ not on this line.
    Now no more than one point of any $\Sigma_j$, $j \neq i$, can be on $l$, else
    $l\subset \Sigma_i \cap \Sigma_j$.  Hence, $p$ has at most two degrees
    of freedom to vary over $\Sigma_j$, and $l$ has two degrees to
    vary over $\Sigma_i$.  Since our choice of $j$ is irrelevant to
    this upper bound, the set of all such $P$ can have at most a
    dimension of four.

    \item \textit{Planes $P$ which intersect each $\Sigma_i$ in a point, but
    without loss of generality, $\Sigma_1 \cap P =\Sigma_2 \cap
    P=p$.}  Since $p$ is fixed by our choice of the special planes,
    $P$ is then determined by two other points, $q$ and $r$, such
    that $p,q,$ and $r$ are not collinear. There are two special
    cases to consider:  First, suppose that $q=\Sigma_3 \cap P$, and
    $r=\Sigma_4 \cap P$, such that $p,q,$ and $r$ are not collinear.
    Since there are two degrees of freedom in which to vary each of
    $q$ and $r$, the collection of planes they span has at most
    dimension 4.  Secondly, if $p, q,$ and $r$ are as above, but
    collinear, then $P$ is determined by an additional point,
    not in any of the special planes.  Since the ambient space is
    $\textbf{P}^5$, we have 5-2=3 degrees in which to vary this
    point.  The ways to get lines $\overline{pq}$ which intersect $\Sigma_4$
    is bounded above by the degrees of freedom we have to vary
    $q$ in its plane, i.e. two.  However, not every line $\overline{pq}$ will
    intersect $\Sigma_4$.  Consider a general line $l_3$ in
    $\Sigma_3$. Denote by $Q_l$ the plane spanned by $p$ and
    $l$.  Note $Q_l$ cannot intersect $\Sigma_4$ in more than a
    point, as if it were to intersect in a line, say $l_4$, then
    $l_3 \cap l_4= \Sigma_3 \cap \Sigma_4$.  Whence, for any line
    $l_3$, there is at most one point $q$ such that $\overline{pq}$
    intersects $\Sigma_4$.  Thus, the planes
    spanned by $\overline{pqr}$, and a point outside the planes has at most
    dimension 3+1=4.

\end{itemize}
\end{proof}

In fact, for this example such a plane can be found explicitly. Let
$\Sigma\subset \pp^5 $ be the plane given by the system of
equations:

\begin{itemize}
\item[] \begin{center}$x_2+x_4=0$\end{center}
\item[] \begin{center}$x_0+x_1+x_3=0$\end{center}
\item[] \begin{center}$x_0+x_3+x_5=0$\end{center}
\end{itemize}

One can check that $\Sigma$ intersects the $\Sigma_i$ in the
following points, where each $P_i$ is disjoint from $\Sigma_j$ for
$j \neq i$. Furthermore, by inspection it is easily seen that the
points are in general linear position.

\begin{itemize}

    \item[] \begin{center}$P_1=[0,0,1,0,-1,0]$\end{center}
    \item[] \begin{center}$P_2=[0,1,0,-1,0,1]$\end{center}
    \item[] \begin{center}$P_3=[1,0,0,0,0,-1]$\end{center}
    \item[] \begin{center}$P_4=[1,0,0,-1,0,0]$\end{center}
\end{itemize}

\begin{propo} The map $\Phi : X_4 \rightarrow W$, induced by sending the
base $\pp^2$ of $X_4$ to a general plane in $\pp^5$
intersecting $\Sigma_1,...,\Sigma_4$ in single points, is a Mori
embedding.
\end{propo}

\begin{proof} To start, we show that $\Phi$ is an embedding. First,
by the lemma we know there exists a plane $\Sigma$ in $\pp^5$
intersecting the $\Sigma_i$ in the desired manner. Thus, we have an
embedding of the base $\pp^2$ into the base $\pp^5$.
By ~\cite[II.7.15]{h},
 $\Phi$ is also an embedding. By
construction, $\Pic{(X)}=\Pic{(W)}$, and $\Phi$ induces a surjection
between the finitely generated Cox rings.  Thus, $\Phi: X_4
\rightarrow W$ is a Mori embedding by Theorem \ref{prop:embed}.
\end{proof}

    To further illustrate the geometry of the situation,
we now show explicitly that $\overline{\NE}^1(W)$ restricts to $\overline{\NE}^1(X)$ for
the above embedding. To this end, consider the six $T$-invariant
hyperplanes in $\pp^5$ given by $P_i:=\{x_i=0\}$.  Note that:

\begin{itemize}

    \item[] \begin{center}$P_0\supseteq \Sigma_1,\Sigma_4$\end{center}
    \item[] \begin{center}$P_1\supseteq \Sigma_1,\Sigma_2$\end{center}
    \item[] \begin{center}$P_2\supseteq \Sigma_1,\Sigma_3$\end{center}
    \item[] \begin{center}$P_3\supseteq \Sigma_2,\Sigma_4$\end{center}
    \item[] \begin{center}$P_4\supseteq \Sigma_2,\Sigma_3$\end{center}
    \item[] \begin{center}$P_5\supseteq \Sigma_3,\Sigma_4$\end{center}

\end{itemize}

    Letting $E_i$
denote the exceptional divisors corresponding to the respective
$\Sigma_i$, we have found 10 $T$-invariant divisors:

\begin{itemize}

    \item[] \begin{center}$D_0:=\pi^*P_0-E_1-E_4$\end{center}
    \item[] \begin{center}$D_1:=\pi^*P_1-E_1-E_2$\end{center}
    \item[] \begin{center}$D_2:=\pi^*P_2-E_1-E_3$\end{center}
    \item[] \begin{center}$D_3:=\pi^*P_3-E_2-E_4$\end{center}
    \item[] \begin{center}$D_4:=\pi^*P_4-E_2-E_3$\end{center}
    \item[] \begin{center}$D_5:=\pi^*P_5-E_3-E_4$\end{center}
    \item[] \begin{center}$E_1,E_2,E_3,E_4$\end{center}

\end{itemize}

    The effective cone of divisors of a toric variety is generated
by classes of $T$-invariant divisors, and from $\Delta^1(W)$ we know
there are ten in this case. Now let $d_i$ represent the restriction
of $D_i$ to $X_4$.  Then $d_i=\pi^*l-l_j-l_k$, where $l$ is a line and
$l_j,l_k$ are the exceptional divisors in $X$ corresponding to the
points $p_j,p_k$. The $E_i$ restrict to the $l_i$. Since these are
the ten exceptional curves of $X$, by Batyrev and Popov's result ~\cite{bp},
these are exactly the generators of $\overline{\NE}^1(X)$.


\bibliographystyle{plain}

\end{document}